\documentclass[a4paper,10pt]{amsart}
\usepackage[arrow,matrix]{xy}
\usepackage{amsmath,amssymb,amscd,bbm,amsthm,mathrsfs,dsfont,hyperref}
\theoremstyle{plain} \textwidth=31pc \textheight=51pc

\usepackage{verbatim}
\numberwithin{equation}{section}
\def\eref #1{(\ref{#1})}

\theoremstyle{definition}
\newtheorem{definition}{Definition}[section]
\newtheorem{proposition}[definition]{Proposition}
\newtheorem{lemma}[definition]{Lemma}
\newtheorem{theorem}[definition]{Theorem}
\newtheorem{example}[definition]{Example}
\newtheorem{corollary}[definition]{Corollary}
\newtheorem{remark}[definition]{Remark}

\newcommand{\rmnum}[1]{\romannumeral #1}
\newcommand{\Rmnum}[1]{\expandafter\@slowromancap\romannumeral #1@}

\makeatletter
\def\rightharpoondownfill@{%
    \arrowfill@\relbar\relbar\rightharpoondown}
\def\rightharpoonupfill@{%
    \arrowfill@\relbar\relbar\rightarrow}
\def\leftharpoondownfill@{%
    \arrowfill@\leftarrow\relbar\relbar}
\def\leftharpoonupfill@{%
    \arrowfill@\leftharpoonup\relbar\relbar}
\newcommand{\xrightharpoondown}[2][]{%
    \ext@arrow 0359\rightharpoondownfill@{#1}{#2}}
\newcommand{\xrightharpoonup}[2][]{%
    \ext@arrow 0359\rightharpoonupfill@{#1}{#2}}
\newcommand{\xleftharpoondown}[2][]{%
    \ext@arrow 3095\leftharpoondownfill@{#1}{#2}}
\newcommand{\xleftharpoonup}[2][]{%
    \ext@arrow 3095\leftharpoonupfill@{#1}{#2}}
\newcommand{\xleftrightharpoons}[2][]{\mathrel{%
    \raise.22ex\hbox{%
        $\ext@arrow 3095\leftharpoonupfill@{\phantom{#1}}{#2}$}%
    \setbox0=\hbox{%
        $\ext@arrow 0359\rightharpoondownfill@{#1}{\phantom{#2}}$}%
    \kern-\wd0 \lower.22ex\box0}%
}
\newcommand{\xrightleftharpoons}[2][]{\mathrel{%
    \raise.25ex\hbox{%
        $\ext@arrow 3095\rightharpoonupfill@{\phantom{#1}}{#2}$}%
    \setbox0=\hbox{%
        $\ext@arrow 0359\leftharpoondownfill@{#1}{\phantom{#2}}$}%
    \kern-\wd0 \lower.25ex\box0}%
}

\makeatother

\title[Castelnuovo-Mumford Regularity and AS-regular Algebras]{\bf Non-commutative Castelnuovo-Mumford Regularity and AS-regular Algebras}

\author{Z. -C. Dong}
\address{School of Mathematical Sciences, Fudan University, Shanghai 200433, China}
\email{dongzhicheng@msn.com}
\author{Q. -S. Wu}
\address{School of Mathematical Sciences, Fudan University, Shanghai 200433, China}
\email{qswu@fudan.edu.cn}

\date{}

\begin{document}
\begin{abstract} Let $A$ be a connected graded $k$-algebra
with a balanced dualizing complex. We prove that $A$ is a Koszul
AS-regular algebra if and only if that the Castelnuovo-Mumford
regularity and the Ext-regularity coincide for all finitely
generated $A$-modules. This can be viewed as a non-commutative
version of \cite[Theorem 1.3]{ro}. By using Castelnuovo-Mumford
regularity, we prove that any Koszul standard AS-Gorenstein algebra
is AS-regular. As a preparation to prove the main result, we also
prove the following statements are equivalent: (1) $A$ is
AS-Gorenstein; (2) $A$ has finite left injective dimension; (3) the
dualizing complex has finite left projective dimension. This
generalizes \cite[Corollary 5.9]{mori}.
\end{abstract}




\maketitle

\section{Introduction}

A coherent sheaf $\mathcal{M}$ on $\mathds{P}^n$ is called
$m$-regular if $H^i(\mathcal{M}(m-i))=0$ for all $i \geq 1$. Mumford
proved a vanishing theorem \cite{mu}: if  $\mathcal{M}$ is
$m$-regular, then it is $\overline{m}$-regular for all $\overline{m}
> m$. Motivated by Mumford's vanishing theorem, a notion of regularity
was introduced by Eisenbud and Goto \cite{eg} for graded modules
over polynomial algebras. This notion of regularity is closely
related to the regularity of sheaves, also to the existence of
linear free resolutions of the truncations of graded modules and to
the degrees of generators of the syzygies. If $A=k[x_1, x_2, \cdots
, x_{n+1}]$, so that Proj $A=\mathds{P}^n$, and $M$ is an $A$-module
of the form $\oplus_{i \geq 0}H^0(\mathds{P}^n, \mathcal{O}_X(i))$
for some scheme $X \subset \mathds{P}^n,$ then the regularity of $M$
is the regularity of $X \subset \mathds{P}^n$ in the sense of
Castelnuovo, studied by Mumford \cite{mu} and many other authors in
literature. There are several competing definitions of
Castelnuovo-Mumford regularity for a graded module over a
commutative or non-commutative connected graded $k$-algebra $A$,
say, CM.reg${}_A M$ defined by using local cohomology (Definition
\ref{cm} and Remark \ref{cm local regularity}), Ext.reg${}_A M$
defined by Ext-group (Definition \ref{ext}) and Tor.reg${}_A M$
defined by Tor-group (Remark \ref{ext tor regularity}). The
Ext-regularity and Tor-regularity of a finitely generated graded
module are always the same. If $A$ is a polynomial algebra with
standard grading, Eisenbud and Goto \cite{eg} proved that
$\mathrm{CM.reg}\ M = \mathrm{Tor.reg}\ M$ for all non-zero finitely
generated graded $A$-modules. R\"{o}mer \cite{ro} proved that the
converse is true, i.e., if $A$ is a commutative connected graded
$k$-algebra generated in degree $1$, such that $\mathrm{CM.reg}\ M =
\mathrm{Tor.reg}\ M$ holds for all non-zero finitely generated
graded $A$-modules, then $A$ is a polynomial algebra with standard
grading (see Theorem \ref{tim romer 4.1}). The non-commutative
Castelnuovo-Mumford regularity was first studied by J${\o}$rgensen
\cite{jo3,jo4}. J${\o}$rgensen proved a version of Mumford's
vanishing theorem for non-commutative projective schemes
\cite[Theorem 2.4 and Corollary 3.3]{jo3}. Under the assumption that
$A$ is a Noetherian  connected graded $k$-algebra having a balanced
dualizing complex, J${\o}$rgensen proved  that for any finitely
generated non-zero $A$-module $M$, $- \mathrm{CM.reg}\ A \leq
\mathrm{Ext.reg}\ M - \mathrm{CM.reg}\ M \leq \mathrm{Ext.reg}\ k$
\cite[Theorem 2.5 ,2.6]{jo4}. In the commutative case, this was also
proved by R\"{o}mer \cite[Theorem 1.2]{ro}. For a Noetherian
connected graded $k$-algebra $A$, by a result of Van den Bergh
\cite[Theorem 6.3]{vdb}, $A$ has a balanced dualizing complex if and
only if that $A$ has finite left and right local cohomology
dimension (see Definition \ref{local cohomology}) and $A$ satisfies
the left and right $\chi$-condition (see Definition \ref{x
condition}). The $\chi$-condition is trivially true when $A$ is
commutative. If $A$ is commutative, then the Krull dimension of $A$
is finite. By a theorem of Grothendieck (see \cite[Theorem
3.5.7]{bh}), the local cohomology dimension of $A$ is finite. These
conditions are very natural and important in non-commutative
projective geometry \cite{az}. Under these assumptions, in this
article, we will prove a non-commutative version (Theorem \ref{main
theorem}) of  R\"{o}mer's result \cite[Theorem 1.3]{ro}.

\noindent {\bf Theorem.} Let $A$ be a Noetherian  connected graded
$k$-algebra having a balanced dualizing complex. Then $A$ is Koszul
AS-regular if and only if $\mathrm{CM.reg}\ M=\mathrm{Ext.reg}\ M$
holds for all non-zero finitely generated graded $A$-modules.

Our proof is different from  R\"{o}mer's in the commutative case.
R\"{o}mer used the following fact in his proof: every commutative
Noetherian  connected graded $k$-algebra can be viewed as a graded
quotient of some polynomial algebra. However, in the non-commutative
case, it is still an unsolved problem whether a Noetherian connected
graded $k$-algebra with a balanced dualizing complex can be viewed
as a graded quotient of an AS-Gorenstein algebra. Actually we will
prove that (\rmnum 3)$\Leftrightarrow$(\rmnum 4)
$\Leftrightarrow$(\rmnum 5) in \cite[Theorem 4.1]{ro} (see also
Theorem \ref{tim romer 4.1}) holds in the non-commutative case. Note
that (\rmnum 1)$\Leftrightarrow$(\rmnum 2) is always true. However,
(\rmnum 1)$\Leftrightarrow$(\rmnum 3) does not hold in general (see
Remark \ref{koszul AS-regular algebras}). Since in the commutative
case, the following crucial fact is used in the R\"{o}mer's proof:
Let $A$ be a commutative Noetherian  connected graded $k$-algebra
generated in degree $1$, then  $A$ is Koszul if and only if
$\mathrm{Ext.reg}\ k<\infty$. This was first a conjecture in
\cite{ae}, and finally proved by L. L. Avramov and I. Peeva in
\cite{ap}. However, this fact is not true in the non-commutative
case since $\mathrm{Ext.reg}\ k<\infty$ holds for those non-Koszul
AS-regular algebras (see Proposition \ref{tim romer theorem 4.2} and
Remark \ref{koszul AS-regular algebras}).

Before giving the proof of the main result, we prove the following
result (Theorem \ref{standard Koszul AS-Gorenstein}) by using
Catelnuovo-Mumford regularity, which is of independent interest.

\noindent {\bf Theorem.} Any Koszul standard AS-Gorenstein algebra
is AS-regular.

An AS-Gorenstein algebra of type $(d, l)$ is called standard if
$d=l$ (Definition  \ref{AS-Gorenstein algebra}).

We also generalize a result of Mori {\cite[Corollary 5.9]{mori}} as
a preparation to the proof of the main result (Theorem \ref{Mori's
Theorem}).

\noindent {\bf Theorem.} Let $A$ be a connected graded $k$-algebra
with a balanced dualizing complex. Then  $A$ is AS-Gorenstein if and
only if that $A$ has finite left injective resolution, if and only
if that the dualizing complex has finite left projective dimension.

\section{Notations and Preliminaries}
Throughout this article, $k$ is a fixed field.
A $k$-algebra $A$ is called an {\it $\mathbb N$-graded $k$-algebra}
if $A$, as a $k$-vector space, has the form $A=\bigoplus_{i\geq
0}A_i$ such that $A_iA_j\subseteq A_{i+j}$ for all $i, j \geq 0.$ If
further, $A_0=k$, then $A$ is called {\it connected graded}. A left
$A$-module $M$ is called a {\it graded $A$-module} if $M$, as a
$k$-vector space, has the form $M=\bigoplus_{i=-\infty}^{\infty}M_i$
such that $A_iM_j\subseteq M_{i+j}$ for all $i\geq 0$ and $j \in
\mathbb Z.$ Right graded $A$-modules are defined similarly.

The opposite algebra of $A$ is denoted by $A^{\mathrm{o}}$; it is
the same as $A$ as $k$-vector spaces, but the product is given by $a
\cdot b = ba$. A right graded $A$-module can be identified with a
left graded $A^{\mathrm{o}}$-module. Unless otherwise stated, we are
working with left modules. We use the term Noetherian for two-sided
Noetherian.

Let $A=\bigoplus_{i\geq 0}A_i$ be an $\mathbb N$-graded $k$-algebra.
We write $\mathfrak m=\bigoplus_{i\geq 1}A_i$. When $A$ is connected
graded, the graded module ${}_Ak \cong A/\mathfrak m$ is called the
{\it trivial module}. A graded module
$M=\bigoplus_{i=-\infty}^{\infty}M_i$ is called {\it left-bounded}
if $M_i=0$ for $i\ll 0$; {\it right-boundedness} and {\it
boundedness} are defined similarly. $M$ is called {\it locally
finite} if each graded piece $M_i$ is a finite dimensional
$k$-vector space.  For any fixed $n \in \mathbb Z$, we use the
notations $M_{>n}=\bigoplus_{i=n+1}^{\infty}M_i$,
$M_{<n}=\bigoplus_{i=-\infty}^{n-1}M_i$ and $M_{\neq
n}=\bigoplus_{i\neq n}M_i$. For any graded $A$-module $M$, the
$n$-th {\it shift} $M(n)$ of $M$, is defined by $M(n)_i=M_{n+i}$.

Let $M$ and $N$ be graded $A$-modules. A left $A$-module
homomorphism $f: M \rightarrow N$ is said to be {\it a graded
homomorphism of degree $l$} if $f(M_i)\subseteq N_{i+l}$ for all
$i\in \mathbb Z$. Let $\mathrm{Gr}\,A$  denote the category of
graded left $A$-modules and graded homomorphisms of degree $0$. Let
$\mathrm{gr}\,A$ be the full subcategory of $\mathrm{Gr}\,A$
consisting of finitely generated graded $A$-modules. Let
$\mathrm{Hom}_{\mathrm{Gr}\,A}(-, -)$ be the homo-functor in the
category $\mathrm{Gr}\,A$.


Let $A$ and $B$ be $\mathbb N$-graded $k$-algebras.  The category
$\mathrm{Gr}\,A^{\mathrm{o}}$ is identified with the category of
graded right $A$-modules; the category
$\mathrm{Gr}\,A\otimes_kB^{\mathrm o}$ is identified with the
category of graded $A$-$B$-bimodules. In particular, the category of
graded $A$-$A$-bimodules is denoted by $\mathrm{Gr}\,A^e$, where
$A^e=A\otimes_k A^{\mathrm{o}}$. The natural restriction functors
are denoted by
$$\mathrm{res}_A: \mathrm{Gr}\ A\otimes_kB^{\mathrm o}\longrightarrow \mathrm{Gr}\ A$$and
$$\mathrm{res}_{B^{\mathrm o}}: \mathrm{Gr}\ A\otimes_kB^{\mathrm o}\longrightarrow \mathrm{Gr}\ B^{\mathrm o}.$$

We define the graded Hom-functor by
$$\mathrm{\underline{Hom}}_A(M, N)=\bigoplus\limits_{n=-\infty}^{\infty}\mathrm{Hom}_{\mathrm{Gr\,A}}(M, N(n)).$$
$\mathrm{\underline{Hom}}_A(M, N)$ is naturally a graded left
$B$-module if $M$ is a graded $A$-$B$-bimodule and a graded right
$C$-module if $N$ is a graded $A$-$C$-bimodule.

For any  $M\in \mathrm{Gr}\,A\otimes_kB^{\mathrm o}$,
$M^\prime=\mathrm{\underline{Hom}}_k(M ,k)$ is called its {\it
Matlis dual}. By definition, $M^\prime$ becomes a graded
$B$-$A$-bimodule. Thus there is an exact contravariant  functor from
$\mathrm{Gr}\,A \otimes_k B^{\mathrm o}$ into $\mathrm{Gr}\,B
\otimes_k A^{\mathrm{o}}$. If $M$ is locally finite, then
$M^{\prime\prime}\cong M$ as $A$-$B$-bimodules.

Let $X^\cdot, Y^\cdot$ be cochain complexes of graded left
$A$-modules. The $i$-th cohomology module of $X^\cdot$ is denoted by
$h^i(X^\cdot)$. A morphism of complexes  $f: X^\cdot \longrightarrow
Y^\cdot$ is called a {\it quasi-isomorphism} if $h^i(f):
h^i(X^\cdot) \longrightarrow h^i(Y^\cdot)$ is an isomorphism for all
$i\in \mathbb Z$. Shifting of complexes is denoted by $[ \ ]$ so
that $(X^\cdot[n])^p=X^{n+p}$. We denote by
$Z^n(X^\cdot)=\mathrm{Ker}(d_{X^\cdot}^n)$ and
$B^n(X^\cdot)=\mathrm{Im}(d_{X^\cdot}^{n-1})$. Let $X^{\geq n}$
denote the (brutal) truncated complex
$$0\longrightarrow X^n \longrightarrow X^{n+1} \longrightarrow \cdots,$$
and let $X^{\leq n}$ denote the complex
$$\cdots \longrightarrow X^{n-1} \longrightarrow X^{n} \longrightarrow 0.$$

The homotopy category of $\mathrm{Gr}\,A$ is denoted by
$K(\mathrm{Gr}\,A)$ and the derived category of $\mathrm{Gr}\,A$ is
denoted by $\mathrm{D}(\mathrm{Gr}\,A)$. For any $X^\cdot \in
\mathrm{D}(\mathrm{Gr}\,A)$, we define
$$\sup X^\cdot = \sup\{i\ |\ h^i(X^\cdot)\neq 0\}~\mathrm{and}~
\inf X^\cdot = \inf\{i\ |\ h^i(X^\cdot)\neq 0\}.$$ When $\sup
X^\cdot<\infty (\mathrm{resp}. \ \inf X^\cdot>-\infty)$, $X^\cdot$
is said to be {\it bounded above} (resp. {\it bounded below}).
$X^\cdot$ is said to be {\it bounded} if $X^\cdot$ is both bounded
above and bounded below. In this case, amp\,$X^\cdot =\sup X^\cdot -
\inf X^\cdot $ is called the {\it amplitude} of $X^\cdot.$ There are
various full subcategories $\mathrm{D}^*(\mathrm{Gr}\,A)$ of
$\mathrm{D}(\mathrm{Gr}\,A)$, where $*=-, +, b$, consisting of
bounded above, bounded below and bounded complexes respectively.
There are also full subcategories $\mathrm{D}_*(\mathrm{Gr}\,A)$ of
$\mathrm{D}(\mathrm{Gr}\,A)$, where $*=$ fg, lf, consisting of
complexes with finitely generated cohomologies and locally finite
cohomologies respectively.  Super- and subscripts are combined
freely.

Let $A$, $B$, and $C$ be $\mathbb N$-graded $k$-algebras. For any
$X^\cdot\in K(A\otimes_k B^{\mathrm o})$ and any $Y^\cdot\in
K(A\otimes_k C^{\mathrm o})$, $\mathrm{Hom}_A^{\cdot}(X^{\cdot},
Y^{\cdot})$ is a complex in $K(B\otimes_k C^{\mathrm o})$ where the
$n$-th term is
$$\mathrm{Hom}_A^{n}(X^{\cdot}, Y^{\cdot})=\prod_p\mathrm{\underline{Hom}}_A^{n}(X^{p}, Y^{p+n}),$$
and the differential $d_{\mathrm{Hom}_A^{\cdot}(X^{\cdot},
Y^{\cdot})}^n$ is
$$(f_p)_p \rightsquigarrow (f_{p+1}\cdot d_{X^{\cdot}}^{p} - (-1)^n d_{Y^{\cdot}}^{p+n}\cdot f_p)_p.$$
This induces a bi-$\partial$-functor,
$$\mathrm{Hom}_A^{\cdot}(-, -):\quad K(A\otimes_k B^{\mathrm o})^{\mathrm o}\times K(A\otimes_k C^{\mathrm o})
\longrightarrow K(B\otimes_k C^{\mathrm o}).$$
The right-derived functor of $\mathrm{Hom}^\cdot$ is denoted by
$R\mathrm{\underline{Hom}}$.

Similarly, for any $X^\cdot\in K(B\otimes_k A^{\mathrm o})$ and
$Y^\cdot \in K(A\otimes_k C^{\mathrm o})$, $X^\cdot\otimes_A Y^\cdot
\in K(B\otimes_k C^{\mathrm o})$, where the $n$-th component
$X^\cdot\otimes_A^n Y^\cdot$ of  $X^\cdot\otimes_A Y^\cdot $ is the
$k$-vector subspace generated by
$$\{x^p \otimes y^q \,|\, x^p \in X^p, y^q \in Y^q, p+q=n \}$$
and the tensor differential $d_{X^\cdot\otimes_A Y^\cdot} =
d_{X^\cdot} \otimes \textrm{id}_Y + \textrm{id}_X \otimes
d_{Y^\cdot}$ (with Koszul sign rule).
This induces a bi-$\partial$-functor,
$$-\otimes_A-:\quad K(B\otimes_k A^{\mathrm o})\times K(A\otimes_k C^{\mathrm o})
\longrightarrow K(B\otimes_k C^{\mathrm o}).$$
The left derived functor of $\otimes$ is denoted by $^L\otimes$.

The Ext and Tor are defined as
$$\mathrm{\underline{Ext}}_A^{i}(X^\cdot, Y^\cdot) =
h^i(R\mathrm{\underline{Hom}}_A(X^\cdot, Y^\cdot)) \quad \textrm
{and} \quad \mathrm{Tor}_i^A(X^\cdot, Y^\cdot) = h^{-i}(X^\cdot
{^L\otimes_A} Y^\cdot).$$

\begin{definition}\label{pd}
Let $A$ be an $\mathbb N$-graded $k$-algebra and let $X^\cdot\in
\mathrm{D}^{-}(\mathrm{Gr}\,A)$. A projective resolution of
$X^\cdot$ is a complex $P^\cdot$ consisting of projective modules
such that there is a quasi-isomorphism $f: P^\cdot
\stackrel{\simeq}{\longrightarrow} X^\cdot$. Moreover, if
$B^i(P^\cdot)\subseteq \mathfrak mP^i$ for each $i\in \mathbb Z$,
then $P^\cdot$ is called a minimal projective resolution of
$X^\cdot$.

We call
$$\mathrm{pd}_A(X^\cdot)=\inf_{P^\cdot}(-\inf\{i\ |P^i\neq 0\})$$
the projective dimension of $X^\cdot$, where the infimum is taken
over all projective resolutions $P^\cdot$ of $X^\cdot$.
\end{definition}

If $A$ is left Noetherian and  $X^\cdot\in
\mathrm{D}_{fg}^{-}(\mathrm{Gr}\,A)$, then
$$\mathrm{pd}_A(X^\cdot)=\sup \{i\,|\, \mathrm{\underline{Ext}}_A^i(X^\cdot, M) \neq 0 \, \text{for some}\,M \in \mathrm{gr}\, A\}. $$
If $P^\cdot \stackrel{\simeq}{\longrightarrow} X^\cdot$ is a minimal
free resolution of $X^\cdot$, then $\mathrm{pd}_A(X^\cdot)=-
\inf\{i\,|\,P^i \neq 0\}.$

\begin{definition}\label{id}
Let $A$ be an $\mathbb N$-graded $k$-algebra and let $X^\cdot\in
\mathrm{D}^{+}(\mathrm{Gr}\ A)$. An injective resolution of
$X^\cdot$ is a complex $I^\cdot$ consisting of injective modules
such that there is a quasi-isomorphism $f: X^\cdot
\stackrel{\simeq}{\longrightarrow} I^\cdot$. Moreover, if
$Z^i(I^\cdot)$ is graded essential in $I^i$
for each $i\in \mathbb Z$,
then $I^\cdot$ is called a minimal injective resolution of
$X^\cdot$.

We call
$$\mathrm{id}_A(X^\cdot)=\inf_{I^\cdot}(\sup\{i\ |I^i\neq 0\})$$
the injective dimension of $X^\cdot$, where the infimum is taken
over all injective resolutions $I^\cdot$ of $X^\cdot$.
\end{definition}

It is easy to know that
$$\mathrm{id}_A(X^\cdot)=\sup \{i\,|\, \mathrm{\underline{Ext}}_A^i(M, X^\cdot) \neq 0 \, \text{for some}\,M \in \mathrm{gr}\, A\}. $$

Next we collect some definitions and facts in connected graded ring
theory we need, which are of basic importance not only in this
article.

\begin{definition}{[AZ, Ye]}\label{local cohomology}
Let $A$ be an $\mathbb N$-graded $k$-algebra and $\mathfrak
m=A_{\geq 1}$.

(1) For any $M\in \mathrm{Gr}\,A$, the $\mathfrak m$-torsion
submodule of $M$ is defined to be
$$\mathrm{\Gamma}_{\mathfrak m}(M)=\{x\in M\ |\ A_{\geq n}\cdot x=0,\ \mathrm{for}\ n \gg 0\}.$$
If $\mathrm{\Gamma}_{\mathfrak m}(M)=M$, then $M$ is said to be
$\mathfrak m$-torsion.

(2) $\mathrm{\Gamma}_{\mathfrak m}: \mathrm{Gr}\,A \longrightarrow
\mathrm{Gr}\,A$, $M \mapsto \mathrm{\Gamma}_{\mathfrak m}(M)$, is a
left exact functor. Its right derived functor
$R\mathrm{\Gamma}_{\mathfrak m}$ is defined on the derived category
$\mathrm{D}^{+}(\mathrm{Gr}\ A)$. The $i$-th local cohomology of
$X^\cdot\in \mathrm{D}^{+}(\mathrm{Gr}\ A)$ is defined to be
$$\mathrm{H}_{\mathfrak m}^{i}(X^\cdot)=h^i(R\mathrm{\Gamma}_{\mathfrak m}(X^\cdot)).$$

(3) The local cohomological dimension of a graded $A$-module $M$ is
defined to be
$$\mathrm{lcd}(M)=\sup \{i\ |\ \mathrm{H}_{\mathfrak m}^{i}(M)\neq 0\}.$$

(4) The cohomological dimension of $\mathrm{\Gamma}_{\mathfrak m}$
is defined to be
$$\mathrm{cd}(\mathrm{\Gamma}_{\mathfrak m})=\sup \{\mathrm{lcd}(M)|\ M\in \mathrm{Gr}\,A\}.$$

(5) $\mathrm{cd}(\mathrm{\Gamma}_{\mathfrak m})$ is also called the
left local cohomological dimension of the algebra $A$.
\end{definition}

\begin{definition}{[AZ, 3.2]}\label{x condition}
Let $A$ be a Noetherian  connected graded $k$-algebra. Then $A$ is
said to satisfy the $\chi$-condition if
$\mathrm{\underline{Ext}}_A^{i}(k, M)$ is right bounded for any $M
\in \mathrm{gr}\,A$ and $i\in \mathbb Z$.
\end{definition}

\begin{definition}\label{depth}
Let $A$ be a connected graded $k$-algebra. For any $X^\cdot\in
\mathrm{D}^+(\mathrm{Gr}\,A)$,
$$\mathrm{depth}_A(X^{\cdot}) =\inf R\mathrm{\underline{Hom}}_A(k, X^{\cdot})=\inf \{i\ |\ \mathrm{\underline{Ext}}_A^i(k, X^{\cdot})\neq 0\}.$$
$\mathrm{depth}_A(X^{\cdot})$ is either an integer or $\infty$.
\end{definition}


\begin{lemma}\label{depth and local cohomology}
Let $A$ be a left Noetherian  connected graded $k$-algebra. Then for
any $M\in \mathrm{Gr}\,A$, $\mathrm{depth}_A(M)=\inf \{i\in \mathbb
Z\ |\ \mathrm{H}_{\mathfrak m}^i(M)\neq 0\}$.
\end{lemma}

%
\begin{theorem}{(The Auslander-Buchsbaum formula)}\label{Auslander-Buchsbaum-theorem}
Let $A$ be a left Noetherian  connected graded $k$-algebra
satisfying the $\chi$-condition. Given any $X^\cdot\in
\mathrm{D_{fg}^b}(\mathrm{Gr}\,A)$ with
$\mathrm{pd}_A(X^\cdot)<\infty$, one has
$$\mathrm{pd}_A(X^\cdot)+\mathrm{depth}_A(X^\cdot)=\mathrm{depth}_A(A).$$
\end{theorem}
\begin{proof}[Proof]
See \cite[Theorem 3.2]{jo1}.
\end{proof}

\begin{definition}{[Ye, 3.3]}\label{dualizing complex}
Let $A$ be a Noetherian connected graded $k$-algebra. A complex
$R^\cdot \in \mathrm{D}^{\mathrm{b}}(\mathrm{Gr}\,A^e)$ is called a
dualizing complex if it satisfies the following conditions:

(1) $\mathrm{id}_A(R^\cdot)<\infty$ \,  and \,
$\mathrm{id}_{A^{\mathrm{o}}}(R^\cdot)<\infty$;

(2) res$_A (R^\cdot)\in
\mathrm{D}_{\mathrm{fg}}^{\mathrm{b}}(\mathrm{Gr}\,A)$ \, and \,
res$_{A^{\mathrm{o}}} (R^\cdot)\in
\mathrm{D}_{\mathrm{fg}}^{\mathrm{b}}(\mathrm{Gr}\,A^{\mathrm{o}})$;

(3) The natural morphisms $A \rightarrow
R\mathrm{\underline{Hom}}_A(R^\cdot,R^\cdot)$ and $A \rightarrow
R\mathrm{\underline{Hom}}_{A^{\mathrm{o}}}(R^\cdot,R^\cdot)$ are
isomorphisms in $\mathrm{D}^{\mathrm{b}}(\mathrm{Gr}\,A^e)$.
\end{definition}

\begin{definition}{[Ye, 4.1]}\label{balanced dualizing complex}
Let $A$ be a Noetherian  connected graded $k$-algebra and
$R^\cdot\in \mathrm{D}^{\mathrm{b}}(\mathrm{Gr}\,A^e)$ be a
dualizing complex over $A$. If there are isomorphisms
$R\mathrm{\Gamma}_{\mathfrak m}(R^\cdot)\cong
R\mathrm{\Gamma}_{\mathfrak m^{\mathrm o}}(R^\cdot)\cong A^\prime$
in $\mathrm{D}(\mathrm{Gr}\,A^e)$, then $R^\cdot$ is called
balanced.

%
\end{definition}

\begin{theorem}{[VdB, 6.3]}\label{vdb}
Let $A$ be a Noetherian  connected graded $k$-algebra. Then $A$ has
a balanced dualizing complex if and only if the following two
conditions are satisfied:

(1) $\mathrm{cd}(\mathrm{\Gamma}_{\mathfrak m}) <\infty$, \, and \,
$\mathrm{cd}(\mathrm{\Gamma}_{\mathfrak m^{\mathrm o}})<\infty$;

(2) Both $A$ and $A^{\mathrm o}$ satisfy the $\chi$-condition.

If these conditions are satisfied, then the balanced dualizing
complex over $A$ is given by $R\mathrm{\Gamma}_{\mathfrak
m}(A)^\prime$.
\end{theorem}

\begin{theorem}{(The local duality theorem)}\label{local duality}
Let $A, C$ be  connected graded $k$-algebras. Assume that $A$ is
left Noetherian with $\mathrm{cd}(\mathrm{\Gamma}_{\mathfrak
m})<\infty$. Then for any $X^\cdot\in
\mathrm{D}(\mathrm{Gr}\,A\otimes_k C^{\mathrm o})$, there is an
isomorphism
$$R\mathrm{\Gamma}_{\mathfrak m}(X^\cdot)^\prime\cong R\mathrm{\underline{Hom}}_A(X^\cdot, R\mathrm{\Gamma}_{\mathfrak m}(A)^\prime)$$
in $\mathrm{D}(\mathrm{Gr}\,C\otimes_k A^{\mathrm o})$.
\end{theorem}
\begin{proof}[Proof]
See \cite[Theorem 5.1]{vdb}.
\end{proof}

Our basic reference for homological algebra is \cite{we}.

\section{AS-Gorenstein algebras}

As a preparation to prove the main result (Theorem \ref{main
theorem}), we generalize a result of Mori \cite[Corollary 5.9]{mori}
in this section. Let's recall some definitions.


\begin{definition}\label{AS-Gorenstein algebra}
Let $A$ be a Noetherian  connected graded $k$-algebra. $A$ is called
left AS-Gorenstein (AS stands for Artin-Schelter) if

(1) $\mathrm{id}_A A = d < \infty;$

(2) $\mathrm{\underline{Ext}}_A^{i}(k, A) = \left\{
\begin{array}{ll}
0& i\neq d\\
k(l)& i=d
\end{array}
\right.$ for some $l\in \mathbb Z$.\\
\end{definition}

Right AS-Gorenstein algebras are defined similarly. $A$ is
AS-Gorenstein means that $A$ is both left and right AS-Gorenstein.

\begin{definition}\label{AS-Cohen-Macaulay}
Let $A$ be a Noetherian  connected graded $k$-algebra.

(1) $A$ is called AS-Cohen-Macaulay if $R\mathrm{\Gamma}_{\mathfrak
m}(A)$ is concentrated in one degree.

(2) An $A$-$A$-bimodule $\omega_A$ is called a balanced dualizing
module if $\omega_A[d]$ is a balanced dualizing complex over $A$ for
some integer $d$.

(3) $A$ is called balanced Cohen-Macaulay if it has a balanced
dualizing module.
\end{definition}

\begin{lemma}\label{balnaced cohen-macaulay}
Let $A$ be a Noetherian connected graded $k$-algebra with a balanced
dualizing complex $R^\cdot$. If pd$_A R^\cdot <\infty,$ then $A$ is
AS-Cohen-Macaulay and balanced Cohen-Macaulay.
\end{lemma}

\begin{proof}[Proof]
Let $F^\cdot\stackrel{\simeq}{\longrightarrow} R^\cdot$ be a
finitely generated minimal free resolution of $R^\cdot \in
\mathrm{D_{fg}^b}(\mathrm{Gr}\,A)$.
Then $$\mathrm{pd}_A R^\cdot = -\inf \{p\in \mathbb Z\ |\ F^p\neq
0\} \geq -\inf R^\cdot.$$

On the other hand, by the local duality theorem,
$R\mathrm{\Gamma}_{\mathfrak m}(k)^\prime\cong
R\mathrm{\underline{Hom}}_A(k, R^\cdot)$ (see Theorem \ref{local
duality}). Hence $\mathrm{depth}_A R^\cdot =0$. It follows from
Auslander-Buchsbaum formula (Theorem
\ref{Auslander-Buchsbaum-theorem}) that $-\inf R^\cdot \leq
\mathrm{pd}_A R^\cdot =\mathrm{depth}_A A$. By Lemma \ref{depth and
local cohomology} and $R^\cdot \cong R\mathrm{\Gamma}_{\mathfrak
m}(A)^\prime$, $$\mathrm{depth}_A A =\inf\{j\in \mathbb Z\ |\
\mathrm{H}_{\mathfrak m}^j(A)\neq 0\}=-\sup R^\cdot.$$

Therefore $\sup R^\cdot = \inf R^\cdot$, and so $A$ is
AS-Cohen-Macaulay and balanced Cohen-Macaulay.
\end{proof}

\begin{lemma}\label{wz}
Let $A$ be a Noetherian connected graded $k$-algebra with a
dualizing complex $R^\cdot$. Then for any $X \in
\mathrm{D^{b}_{fg}}(\mathrm{Gr}\,A)$, the following holds.

(1) id$_{A^{\mathrm{o}}} (\mathrm{RHom}_A(X, R)) \leq $ pd$_A \, X
+$ id$_{A^{\mathrm{o}}} \, R;$

(2) pd$_{A^{\mathrm{o}}} (\mathrm{RHom}_A(X, R)) \leq $ id$_A \, X -
\inf R.$

As a consequence, pd$_A \, X < \infty$ if and only if
id$_{A^{\mathrm{o}}} (\mathrm{RHom}_A(X, R)) < \infty.$
\end{lemma}

\begin{proof}[Proof] See \cite[Lemma 2.1]{wz}.
\end{proof}

The following theorem says that if the balanced dualizing complex
has finite projective dimension, then $A$ is AS-Gorenstein. This is
a generation of \cite[Corollary 5.9]{mori},where $A$ is a Noetherian
balanced Cohen-Macaulay algebra. For the completeness, we give a
proof here.

\begin{theorem}\label{Mori's Theorem}
Let $A$ be a Noetherian connected graded $k$-algebra with a balanced
dualizing complex $R^\cdot$. Then the following are equivalent:

(1) $A$ is AS-Gorenstein;

(2) id$_A\, A < \infty;$

(3) pd$_A \, R^\cdot<\infty;$

(4) For any $X \in \mathrm{D^{b}_{fg}}(\mathrm{Gr}\,A)$,
pd$_A \,X  < \infty$ if and only if id$_A \,X <\infty$.
\end{theorem}

\begin{proof}[Proof]
(1) $\Rightarrow$ (4). Suppose that $M \in \mathrm{gr}\,A$ with
pd$_A M < \infty$. Then $M$ has a finitely generated free resolution
of finite length. Since each term of the resolution has finite
injective dimension, so does $M$. By induction on the amplitude, it
is easy to see that pd$_A \,X  < \infty$ implies id$_A \,X <\infty$
for any $X \in \mathrm{D^{b}_{fg}}(\mathrm{Gr}\,A)$.

On the other hand, suppose first that $M \in \mathrm{gr}\,A$ with
id$_A M <\infty$. There is a convergent spectral sequence
$$E_2^{p,q}=\mathrm{\underline{Ext}}_A^p(\mathrm{\underline{Ext}}_{A^{\mathrm{o}}}^{-q}(k, A),
M) \Rightarrow \mathrm{Tor}^A_{-p-q}(k, M).
$$
Since  id$_A A < \infty$, $E_2^{p,q}$ = 0 for $q \ll 0.$ Hence
$\mathrm{Tor}^A_n(k, M) =0$ for $n \gg 0$ and so pd$_A M <\infty$.
Now suppose in general that $X \in
\mathrm{D^{b}_{fg}}(\mathrm{Gr}\,A)$ with id$_A \,X <\infty$. We
have to show that pd$_A \,X  < \infty$. Let $F^\cdot \cong X$ be a
finitely generated free resolution of $X$ and $s= \inf X$. It
suffices to prove the $s$-th syzygy of $F^\cdot$ has finite
projective dimension. Since id$_A A < \infty$ and id$_A \,X
<\infty$, any finitely generated free module has finite injective
dimension, and so the $s$-th syzygy of $F^\cdot$ has finite
injective dimension. The claim follows from the module case.

(4) $\Rightarrow$ (3). Since id$_A R^\cdot <\infty,$ it follows that
pd$_A R^\cdot <\infty.$

(3) $\Rightarrow$ (2). By lemma \ref{balnaced cohen-macaulay},
$\omega_A=H^{-d}(R^\cdot)$ is a balanced dualizing module of $A$,
where pd$_A R^\cdot = d$ and pd$_A(\omega_A) < \infty$. Since
depth$_A(\omega_A)= $depth$_A A$, it follows from
Auslander-Buchsbaum formula (Theorem
\ref{Auslander-Buchsbaum-theorem}) that pd$_A(\omega_A) =0$. So
$\omega_A$ is free. Since id$_A(\omega_A) < \infty$, it follows that
id$_A A < \infty.$

 (2) $\Rightarrow$ (1). Since $A$ has a balanced dualizing complex, $A$ satisfies the
 $\chi$-condition. Since id$_A A < \infty$, $c = \sup \{i\,|\,
\mathrm{\underline{Ext}}_A^i(k, A) \neq 0 \} < \infty.$ It follows
from the double Ext spectral sequence
\begin{displaymath}
E_2^{p,q}=\mathrm{\underline{Ext}}_A^p(\mathrm{\underline{Ext}}_{A^{\mathrm{o}}}^{-q}(k,
A), A) \Rightarrow \left\{
\begin{array}{ll}
0 & p + q \neq 0\\
k & p + q = 0,
\end{array}
\right.
\end{displaymath}
that $d := \mathrm{depth}_{A^{\mathrm{o}}} A < \infty$. Since
$A^{\mathrm{o}}$ satisfies the $\chi$-condition, it follows from the
definition of $c$ that $E_2^{c, -d} = E_{\infty}^{c, -d} \neq 0$. So
$c = d$, that is,
$$\mathrm{depth}_{A^{\mathrm{o}}} A = \sup \{i\,|\,
\mathrm{\underline{Ext}}_A^i(k, A) \neq 0 \}.$$

Since id$_A A < \infty$, by Lemma \ref{wz} pd$_{A^{\mathrm{o}}}
R^\cdot < \infty$. Similar to (3) $\Rightarrow$ (2), we have
id$_{A^{\mathrm{o}}} A < \infty.$ Thus by the left-right symmetric
version of the above spectral sequence, $$\mathrm{depth}_A A = \sup
\{i\,|\, \mathrm{\underline{Ext}}_{A^{\mathrm{o}}}^i(k, A) \neq 0
\}.$$  Hence
$$\mathrm{depth}_A A =
\sup \{i\,|\, \mathrm{\underline{Ext}}_A^i(k, A) \neq 0 \} =
\mathrm{depth}_{A^{\mathrm{o}}} A = \sup \{i\,|\,
\mathrm{\underline{Ext}}_{A^{\mathrm{o}}}^i(k, A) \neq 0 \}.$$
Therefore $A$ is AS Gorenstein.
\end{proof}

The idea in (2) $\Rightarrow$ (1) originates from \cite[Theorem
3.8]{sz}. The above result also tells us that if $A$ is a Noetherian
connected graded $k$-algebra with a balanced dualizing complex, then
$A$ is left AS-Gorenstein if and only if $A$ is right AS-Gorenstein.
This is \cite[Corollary 4.6]{jo5}. Note that the first condition in
Theorem \ref{Mori's Theorem} is two-sided, while the others are
one-sided. In particular, when $A$ has a balanced dualizing complex,
then $A$ has finite left injective dimension if and only if that $A$
has finite right injective dimension.

\section{Castelnuovo-Mumford regularity and AS-Gorenstein algebras }

In this section, we first recall the definitions and some facts of
Castelnuovo-Mumford regularity and Ext-regularity. Then we prove any
Koszul standard AS-Gorenstein algebra is AS-regular (see Definition
\ref{standard}). Note $A$ is a Noetherian connected graded
$k$-algebra as always in this article.

We fix the following conventions: $\inf \{ \varnothing \}=+\infty$,
$\inf \{ \mathbb Z \}=-\infty$.
\begin{definition}{\cite[2.1]{jo4}}\label{cm}
For any $X^{\cdot} \in \mathrm{D}(\mathrm{Gr}\,A)$, the
Castelnuovo-Mumford regularity of $X^{\cdot}$ is defined to be
$$\mathrm{CM.reg}\ X^\cdot = \inf\{p\in \mathbb{Z}\ |\ \mathrm{H}_{\mathfrak m}^i(X^\cdot)_{>p-i}=0,\forall\ i\in \mathbb Z\}.$$
\end{definition}

If $A$ has a balanced dualizing complex and $0\ncong
X^\cdot\in\mathrm{D_{fg}^b}(\mathrm{Gr}\ A)$, then $\mathrm{CM.reg}\
X^\cdot\neq - \infty$ by the local duality theorem (Theorem
\ref{local duality}) and $\mathrm{CM.reg}\ X^\cdot\neq + \infty$ by
Theorem \ref{vdb} (see \cite[Observation 2.3]{jo4}). Moreover, we
have $\mathrm{CM.reg}\ (_AA)=\mathrm{CM.reg}\ (A_A)$ by
\cite[Corollary 4.8]{vdb}.

\begin{remark}\label{cm local regularity}
R\"{o}mer used the notion of local-regularity in \cite[1.1]{ro}. For
any $M \in \mathrm{Gr}\,A$,  $\mathrm{reg}_A^L(M) = \inf \{p\in
\mathbb Z\ |\ \mathrm{H}_{\mathfrak m}^i(M)_{>p-i}=0 ,\forall\ i\geq
0\}$ is called the local-regularity of $M$. This is identical as the
definition of Castelnuovo-Mumford regularity above.
\end{remark}

\begin{example}\label{example1}
Since ${}_Ak$ is $\mathfrak m$-torsion, we have
$$\mathrm{H}_{\mathfrak m}^i({}_Ak)=
\left\{
\begin{array}{ll}
0 & i\neq 0\\
k & i=0.
\end{array}
\right.$$ By definition \ref{cm}, it is easy to see that
$\mathrm{CM.reg}\ {}_Ak=0$.
\end{example}

\begin{definition}{\cite[2.2]{jo4}}\label{ext}
For any $X^{\cdot} \in \mathrm{D}(\mathrm{Gr}\ A)$, the
Ext-regularity of $X^{\cdot}$ is defined to be
$$\mathrm{Ext.reg}\ X^\cdot = \inf\{p\in \mathbb{Z}\ |\ \mathrm{\underline{Ext}}_A^i(X^\cdot, k)_{<-p-i}=0,\forall\ i\in \mathbb Z\}.$$
\end{definition}

\begin{remark}\label{ext tor regularity} There is also a notion  of
regularity defined by Tor-group in literature, which is called the
Tor-regularity as in \cite[1.1]{ro} by R\"{o}mer. For any $M\in
\mathrm{Gr}\,A$, $\mathrm{Tor.reg}_A(M)$ or $\mathrm{reg}_A^T(M) =
\inf \{p\in \mathbb Z\ |\ \mathrm{Tor}_i^A(k_A, M)_{>p+i}=0
,\forall\ i\geq 0\}$ is called the Tor-regularity of $M$. If $M\in
\mathrm{gr}\,A$, then as graded $k$-vector spaces,
$$\mathrm{\underline{Ext}}_A^j(M, k)^{\prime} \cong \mathrm{Tor}_j^A(k_A, M).$$
Hence $\mathrm{Ext.reg}\ M=\mathrm{reg}_A^T(M)$, i.e., the
Tor-regularity is the same as the Ext-regularity for any $M\in
\mathrm{gr}\,A$.
\end{remark}

For all $0 \ncong X^\cdot \in \mathrm{D_{fg}^b}(\mathrm{Gr}\,A)$,
 $\mathrm{Ext.reg}\ X^\cdot\neq -\infty$. If $\mathrm{Ext.reg}\ X \leq p$
 and $F^{\cdot} \cong X^\cdot$ is a minimal free resolution of
 $X^\cdot$, then the generators of $F^i$ are
 concentrated in degree $\leq p-i$.
It is possible that
 $\mathrm{Ext.reg}\ X^\cdot = \infty$.
Note that $\mathrm{Ext.reg}\ (_Ak) = 0$ if and only if $A$ is
Koszul, i.e., ${}_Ak$ has a linear free resolution.
It was conjectured in \cite{ae} and proved in \cite{ap}
 that a commutative connected graded $k$-algebra generated in
 degree $1$ is Koszul if and only if that $\mathrm{Ext.reg}\ k < \infty$.

By using the minimal free resolution, we see that $\mathrm{Ext.reg}\
(_Ak)=\mathrm{Ext.reg}\ (k_A)$.

\begin{example}\label{example2}
Since ${}_AA$ is free,
$$\mathrm{\underline{Ext}}_{A}^i(A, k)=
\left\{
\begin{array}{ll}
0 & i\neq 0\\
k & i=0.
\end{array}
\right.$$ By Definition \ref{ext}, $\mathrm{Ext.reg}\ {}_AA=0$.
\end{example}

The following result was proved in {\cite[Theorem 2.5, 2.6]{jo4}},
which plays a key role in this article.
\begin{theorem}\label{jo}
Let $A$ be a Noetherian  connected graded $k$-algebra with a
balanced dualizing complex. Given any
$X^\cdot\in\mathrm{D_{fg}^b}(\mathrm{Gr}\,A)$ with $X^\cdot\ncong
0$. Then
$$- \,\mathrm{CM.reg}\ A \leq \mathrm{Ext.reg}\ X^\cdot-\mathrm{CM.reg}\ X^\cdot \leq \mathrm{Ext.reg}\ k.$$
\end{theorem}

If further, $A$ is Koszul and $\mathrm{CM.reg}\ A=0$, then
$\mathrm{CM.reg}\ X^\cdot=\mathrm{Ext.reg}\ X^\cdot$ for any $0
\ncong X^\cdot\in\mathrm{D_{fg}^b}(\mathrm{Gr}\ A).$

A left AS-Gorenstein algebra in the Definition \ref{AS-Gorenstein
algebra} is sometimes called of type $(d, l)$. If $A$ is left
AS-Gorenstein of type $(d, l)$ and right AS-Gorenstein, say of type
$(d', l')$, then it is well-known that $d = d'$ and $l = l'$.

\begin{definition}\label{standard}
Let $A$ be a Noetherian  connected graded $k$-algebra.

(1) $A$ is called {\it standard} AS-Gorenstein if $A$ is an
AS-Gorenstein algebra with $l=d$.

(2) $A$ is called AS-regular (Artin-Schelter regular) if $A$ is
AS-Gorenstein and $A$ has finite global dimension.

\end{definition}

%
%

It is an easy fact that any Koszul AS-regular algebra is standard
(see \ref{standard Koszul AS-Gorenstein}). However, a Koszul
AS-Gorenstein algebra is not always standard, e.g. $A=k[x]/(x^2)$ is
a Noetherian Koszul connected graded $k$-algebra which has infinite
global dimension. Moreover, $\mathrm{id}_A(A)=0$ and
$\mathrm{\underline{Ext}}_A^0(k, A) \cong k(-1)$. Thus, $A$ is a
non-standard AS-Gorenstein algebra with $l=-1$ and $d=0$. In the
final part of this section, we prove
that any Koszul standard AS-Gorenstein algebra is AS-regular.

\begin{lemma}\label{CM reg of AS-Gorenstein algebra}
Let $A$ be an AS-Gorenstein algebra of type $(d, l)$. Then
$$\mathrm{CM.reg\ A}=d-l.$$
\end{lemma}

\begin{proof}[Proof]Let
$$0\longrightarrow {_AA}\longrightarrow I^0\longrightarrow I^1\longrightarrow \cdots$$
be a minimal injective resolution of $_AA$. Then
$\mathrm{\underline{Ext}}_A^i(k,A)=\mathrm{\underline{Hom}}_A(k,I^i)$.
By Definition \ref{AS-Gorenstein algebra},
\begin{displaymath}
\mathrm{\underline{Ext}}_A^i(k,A)= \left\{
\begin{array}{ll}
0 & i\neq d\\
k(l) & i=d,
\end{array}
\right.
\end{displaymath}
where $d=\mathrm{id}_A(A)$.

Since $\mathrm{\Gamma}_{\mathfrak m}(I^i) \cong
E(\mathrm{\underline{Hom}}_A(k,I^i))$, then

\begin{displaymath}
\mathrm{H}_{\mathfrak m}^i(A)\cong \left\{
\begin{array}{ll}
0 & i\neq d\\
A^{\prime}(l) & i=d.
\end{array}
\right.
\end{displaymath}
By Definition \ref{cm}, it is obvious that $\mathrm{CM.reg\ A}=d-l$.
\end{proof}

\begin{corollary}\label{reg-equal}
Let $A$ be a Koszul standard AS-Gorenstein algebra. Then for any
$X^\cdot\in\mathrm{D_{fg}^b}(\mathrm{Gr}\ A)$ with $X^\cdot\ncong
0$, $$\mathrm{CM.reg\ X}=\mathrm{Ext.reg\ X}.$$
\end{corollary}

\begin{proof}[Proof] Direct from Theorem \ref{jo} and Lemma \ref{CM reg of AS-Gorenstein algebra}
\end{proof}

\begin{theorem}\label{standard Koszul AS-Gorenstein}
Let $A$ be a Noetherian connected graded $k$-algebra. If $A$ is
Koszul, then the following statements are equivalent:

(1) $A$ is AS-regular;

(2) $A$ is standard AS-Gorenstein.
\end{theorem}

\begin{proof}[Proof]
(2) $ \Rightarrow$ (1). It suffices to prove that $\mathrm{pd}{}_A \
k < \infty$. Since $A$ is a Koszul algebra, ${}_Ak$ has a minimal
free resolution
$$F_\cdot:\quad \cdots\longrightarrow F_i \longrightarrow \cdots \longrightarrow F_0 \longrightarrow k \longrightarrow 0,$$
where $F_i=A(-i)^{\beta_i}$. Suppose that $\mathrm{id}_A A=d.$ We
claim that the $d$-th syzygy $Z_d(F_\cdot)$ of $F_\cdot$ is $0$.
Assume on the contrary that $Z_d(F_\cdot)\neq 0.$ Note that
$$\cdots \longrightarrow F_{d+2} \longrightarrow F_{d+1} \longrightarrow Z_d(F_\cdot) \longrightarrow 0$$
is a minimal free resolution of $Z_d(F_\cdot)$.  Then it is easy to
see that $\mathrm{Ext.reg}\ Z_d(F_\cdot)=d+1$.

Since $\mathrm{\underline{Ext}}_A^{i}(k, A)=0$ for $i\neq d$,
$\mathrm{\underline{Ext}}_A^{i}(Z_{d-1}(F_\cdot), A)=0$ for $i\neq
0$. By \cite[Corollary 4.10]{ye} and \cite[Theorem 1.2]{jo0},
$A_\alpha(-d)[d]$ is a balanced dualizing complex over $A$, where
$\alpha$ is an automorphism of $A$ as a graded $k$-algebra. It
follows from the local duality theorem (Theorem \ref{local duality})
that $\mathrm{H}_{\mathfrak m}^i(Z_{d-1}(F_\cdot))=0$ for $i\neq d$.
Since $A$ is AS-Gorenstein,
 $\mathrm{H}_{\mathfrak m}^i(A)=0$ for $i\neq d$.
It follows from lcd$(A)=d$ and the short exact sequence $$ 0
\longrightarrow Z_d(F_\cdot)\longrightarrow F_d\longrightarrow
Z_{d-1}(F_\cdot)\longrightarrow 0 $$ that $\mathrm{H}_{\mathfrak
m}^i(Z_d(F_\cdot))=0$ for $i\neq d$ and the following sequence is
exact:
\begin{equation}\label{1}
0\longrightarrow \mathrm{H}_{\mathfrak
m}^d(Z_d(F_\cdot))\longrightarrow \mathrm{H}_{\mathfrak m}^d(F_d)
\longrightarrow \mathrm{H}_{\mathfrak
m}^d(Z_{d-1}(F_\cdot))\longrightarrow 0.
\end{equation}
Since $A$ is a standard AS-Gorenstein algebra, $\mathrm{CM.reg}\
A=0$ by Lemma \ref{CM reg of AS-Gorenstein algebra}.  Hence
$\mathrm{CM.reg}\ F_d=d$ by Definition \ref{cm}. Therefore
$\mathrm{H}_{\mathfrak m}^d(Z_d(F_\cdot))_{>0}=\mathrm{H}_{\mathfrak
m}^d(F_d)_{>0}=0$ by \eref{1}. This implies that $\mathrm{CM.reg}\
Z_d(F_\cdot)\leq d$, which contradicts to that $\mathrm{Ext.reg}\
Z_d(F_\cdot)=d+1$ by Corollary \ref{reg-equal}. Hence
$Z_d(F_\cdot)=0$, which means that $\mathrm{pd}_A \ k < \infty$.

(1) $\Rightarrow$ (2). Suppose that $A$ is a Koszul AS-regular
$k$-algebra with global dimension $d$. Let $L_\cdot: \,
0\longrightarrow L_d \longrightarrow \cdots \longrightarrow L_0
\longrightarrow k \longrightarrow 0,$ where $L_i=A(-i)^{\beta_i}$,
be a minimal free resolution of ${}_Ak$. Since
$$\mathrm{\underline{Ext}}_A^{i}(k, A)=
\left\{
\begin{array}{ll}
0& i\neq d\\
k(l)& i=d
\end{array}
\right. $$for some $l\in \mathbb Z$, $\mathrm{Hom}_A(L_\cdot, A)$ is
a minimal free resolution of $k_A(l)$. It follows from the Koszulity
of $A$ that $d=l$. Hence $A$ is standard.
\end{proof}

\section{Castelnuovo-Mumford regularity and AS-regular algebras}

If $A$ is a polynomial algebra with standard grading, Eisenbud and
Goto \cite{eg} proved that $\mathrm{CM.reg}\ M = \mathrm{Tor.reg}\
M$ for all non-zero finitely generated graded $A$-modules. Recently,
R\"{o}mer proved that the converse is true \cite[Theorem 4.1]{ro}.
We copy R\"{o}mer's result here for the convenience.

\begin{theorem}\label{tim romer 4.1}
Let $A$ be a commutative Noetherian  connected graded $k$-algebra
generated in degree 1. The following statements are equivalent:

(\rmnum 1) For all $M\in \mathrm{gr}\ A$, $\mathrm{CM.reg}_A
(M)-\mathrm{CM.reg}_A (A)=\mathrm{Ext.reg}_A (M)$;

(\rmnum 2) For all $M\in \mathrm{gr}\ A$, $\mathrm{Ext.reg}_A
(M)=\mathrm{CM.reg}_A (M)+ \mathrm{Ext.reg}_A (k)$;

(\rmnum 3) For all $M\in \mathrm{gr}\ A$, $\mathrm{Ext.reg}_A
(M)=\mathrm{CM.reg}_A (M)$;

(\rmnum 4) $A$ is Koszul and $\mathrm{CM.reg}_A (A)=0$;

(\rmnum 5) $A=k[x_1, \cdots, x_n]$ is a polynomial ring with
standard grading.
\end{theorem}

In this section, we prove a non-commutative version of \cite[Theorem
4.1]{ro}. We start from the following two lemmas.

\begin{lemma}\label{linear free resolution}
Let $X^\cdot \in \mathrm{D^{b}}(\mathrm{Gr}\, A)$ with
$h^i(X^\cdot)_{<-i}=0$ for any $i \in \mathbb Z.$ Then for any $i
\in \mathbb Z,$ $\mathrm{\underline{Ext}}_A^i(X^\cdot,k)_{>-i}=0$ .
\end{lemma}

\begin{proof}[Proof]
If $X^\cdot \cong 0 \in \mathrm{D^{b}}(\mathrm{Gr}\ A)$, there is
nothing to prove. So we assume $X^\cdot \ncong 0.$  We prove the
assertion by induction on the amplitude amp\,$X^\cdot = \sup X^\cdot
- \inf X^\cdot$.

If amp\,$X^\cdot = 0,$ then $h^s(X^\cdot)\neq 0$ and  $X^\cdot\cong
h^s(X^\cdot)[-s] \in \mathrm{D^{b}}(\mathrm{Gr}\,A)$ for $s= \inf
X^\cdot$. Since $h^s(X^\cdot)_{<-s}=0$, $h^s(X^\cdot)[-s]$ has a
minimal free resolution $F^\cdot$, where $(F^i)_{<-i}=0$ for $i \leq
s$ and $F^i=0$ for $i>s$.
Then
$\mathrm{\underline{Ext}}_A^{-i}(X^\cdot,k)_{>i}=\mathrm{\underline{Hom}}_A(F^{i},
k)_{>i}=0$ for $i\leq s$ and
$\mathrm{\underline{Ext}}_A^{-i}(X^\cdot,k)=0$ for $i>s$. Thus the
assertion holds for amp\,$X^\cdot = 0$.

If amp\,$X^\cdot = n>0$, we may assume that $X^j=0$ for either $j<s$
or $j>s+n$, where $s = \inf X^\cdot$. Consider the following exact
triangle in $\mathrm{D^b}(\mathrm{Gr}\,A)$
$$\mathrm{Ker}(d_{X^\cdot}^s)[-s]\stackrel{\lambda}{\longrightarrow}X^\cdot \longrightarrow
\mathrm{cone}(\lambda)\longrightarrow
\mathrm{Ker}(d_{X^\cdot}^s)[-s+1].$$  and the induced long exact
sequence
$$\cdots\longrightarrow \mathrm{\underline{Ext}}_A^i(\mathrm{cone}(\lambda),k)\longrightarrow
\mathrm{\underline{Ext}}_A^i(X^\cdot,k)\longrightarrow
\mathrm{\underline{Ext}}_A^i(\mathrm{Ker}(d_{X^\cdot}^s)[-s],k)\longrightarrow
\cdots.$$ Since amp\,$(\mathrm{Ker}(d_{X^\cdot}^s)[-s])=0$ and
amp\,$(\mathrm{cone}(\lambda))<n$, by induction, it is easy to see
that $\mathrm{\underline{Ext}}_A^i(X^\cdot,k)_{>-i}=0$ for any $i\in
\mathbb Z$.
\end{proof}


Let $A$ be a ring, and $f: F^\cdot \longrightarrow R^\cdot$ be a
morphism of bounded below $A$-module complexes. Set $s=\inf
R^\cdot$. Then $f$ naturally induces a morphism between $F^{\geq s}$
and $R^\cdot$, denoted by $\tilde f$. Define a morphism $g$ between
$F^{\leq s-1}$ and $\mathrm{cone}(\tilde f)$ as follows:
\begin{displaymath}
\xymatrix{
   F^{\leq s-1}\ar[d]_{g}& \cdots \ar[r]^{} & F^{s-2} \ar[d]_{} \ar[r]^{} & F^{s-1} \ar[d]_{} \ar[r]^{} & 0 \ar[d]^{} \ar[r]^{}& \cdots\\
   \mathrm{cone}(\tilde f)& \cdots \ar[r]^{} & R^{s-2} \ar[r]^{} & F^s\oplus R^{s-1} \ar[r]^{} &  F^{s+1}\oplus R^{s} \ar[r]^{} & \cdots}
\end{displaymath}
where
\begin{displaymath}
g^j = \left\{
\begin{array}{ll}
f^j & \quad j\leq s-2\\
0 & \quad j\geq s\\
 \left(
\begin{array}{c}
d_{F^\cdot}^{s-1} \\
f^{s-1}
\end{array}
\right)& \quad j=s-1.
\end{array}
\right.
\end{displaymath}

Then $g: F^{\leq s-1} \longrightarrow  \mathrm{cone}(\tilde f)$ is
indeed a morphism of complexes.

\begin{lemma}\label{quasi-isomorphism} Let the notations be as above.
If $f: F^\cdot \longrightarrow R^\cdot$ is a quasi-isomorphism, then
$g$ is a quasi-isomorphism.
\end{lemma}
\begin{proof}[Proof]
 It is easy to see that
$h^j(F^{\leq s-1}) = 0 = h^j(\mathrm{cone}(\tilde f))$ for all
$j\neq s-1$, since $f$ is a quasi-isomorphism. We are left to prove
that $h^{s-1}(g): h^{s-1}(F^{\leq s-1})\longrightarrow
h^{s-1}(\mathrm{cone}(\tilde f))$ is an isomorphism.

For any  $x^{s-1}+\mathrm{Im}(d_{F^\cdot}^{s-2})\in
\mathrm{Ker}(h^{s-1}(g))$, there exists $y^{s-2}\in R^{s-2}$ such
that
$(d_{F^\cdot}^{s-1}(x^{s-1}),f^{s-1}(x^{s-1}))=(0,d_{R^\cdot}^{s-2}(y^{s-2}))$.
Hence $x^{s-1}\in
\mathrm{Ker}(d_{F^\cdot}^{s-1})=\mathrm{Im}(d_{F^\cdot}^{s-2})$. It
follows that $h^{s-1}(g)$ is injective.

On the other hand, suppose $(x^s,y^{s-1})\in
\mathrm{Ker}(d_{\mathrm{cone}(\tilde f)}^{s-1})$, that is
$(-d_{F^\cdot}^{s}(x^s), -f^s(x^s)+d_{R^\cdot}^{s-1}(y^{s-1}))=0$.
Since $f$ is a quasi-isomorphism, there exists $x^{s-1}\in F^{s-1}$
such that $d_{F^\cdot}^{s-1}(x^{s-1})=x^s$. Then
\begin{eqnarray*}
(x^s,y^{s-1})-g^{s-1}(x^{s-1})&=&(x^s,y^{s-1})-(d_{F^\cdot}^{s-1}(x^{s-1}),
f^{s-1}(x^{s-1}))\\&=& (0, y^{s-1}-f^{s-1}(x^{s-1})).
\end{eqnarray*}
Since
$d_{R^\cdot}^{s-1}(y^{s-1}-f^{s-1}(x^{s-1}))=f^s(x^s)-d_{R^\cdot}^{s-1}f^{s-1}(x^{s-1})=0$,
there exists $x^{s-2}\in R^{s-2}$ such that
$d_{R^\cdot}^{s-2}(x^{s-2})=y^{s-1}-f^{s-1}(x^{s-1})$. Thus
$$h^{s-1}(g)
(x^{s-1}+\mathrm{Im}(d_{F^\cdot}^{s-2}))=(x^s,y^{s-1})+\mathrm{Im}(d_{\mathrm{cone}(\tilde
f)}^{s-2}).$$ It follows that $h^{s-1}(g)$ is surjective.
\end{proof}

The following is the main result in this article, which is a
generalization of (\rmnum 3)$\Leftrightarrow$(\rmnum
4)$\Leftrightarrow$(\rmnum 5) in \cite[Theorem 4.1]{ro} (see also
Theorem \ref{tim romer 4.1}) to the non-commutative case.
\begin{theorem}\label{main theorem}
Let $A$ be a Noetherian  connected graded $k$-algebra with a
balanced dualizing complex. Then the following are equivalent:

(1) $\mathrm{CM.reg}\ M=\mathrm{Ext.reg}\ M$ holds for all $M\in
\mathrm{gr}\,A$.

(2) $A$ is Koszul and $\mathrm{CM.reg}\ A = 0$.

(3) $A$ is a Koszul AS-regular $k$-algebra.
\end{theorem}

\begin{proof}[Proof]

(3) $\Rightarrow$ (2). By Theorem \ref{standard Koszul
AS-Gorenstein},  $A$ is standard. It follows from lemma \ref{CM reg
of AS-Gorenstein algebra}  that $\mathrm{CM.reg}\ A = 0$.

(2) $\Rightarrow$ (1). Direct from Theorem \ref{jo}.


(1) $\Rightarrow$ (3). Let $R^\cdot$ be the balanced dualizing
complex over $A$. Then $R^\cdot\cong R\mathrm{\Gamma}_{\mathfrak
m}(A)^\prime$ by \cite[Theorem 6.3]{vdb}. By assumption,
$\mathrm{CM.reg}\ A=\mathrm{Ext.reg}\ A=0$, so
$\mathrm{H}_{\mathfrak m}^i(A)_{>-i}=0$ for any  $i \in \mathbb Z$
by Definition \ref{cm}.  Thus $h^i(R^\cdot)_{<-i}=0$ for any $i\in
\mathbb Z$. By Lemma \ref{linear free resolution},
$\mathrm{\underline{Ext}}_A^i(R^\cdot,k)_{>-i}=0$ for any $i\in
\mathbb Z$.

On the other hand, since $R^\cdot$ is a balanced dualizing complex
over $A$, $R\mathrm{\Gamma}_{\mathfrak m}(R^\cdot)\cong A^\prime$ by
Definition \ref{balanced dualizing complex}. Therefore,
$\mathrm{CM.reg}\ R^\cdot=0$ by Definition \ref{cm}. Again by
assumption, $\mathrm{Ext.reg}\ k=\mathrm{CM.reg}\ k=0.$ It follows
from Theorem \ref{jo} that $\mathrm{Ext.reg}\
R^\cdot=\mathrm{CM.reg}\ R^\cdot=0$. Thus by Definition \ref{ext},
$\mathrm{\underline{Ext}}_A^i(R^\cdot,k)_{<-i}=0$ for any $i\in
\mathbb Z$.

Since $R^\cdot \in \mathrm{D_{fg}^b}(\mathrm{Gr}\,A)$, $R^\cdot$ has
a finitely generated minimal free resolution
$F^\cdot\stackrel{\simeq}{\longrightarrow} R^\cdot$.
Set $s= \inf\{j\ |\ h^j(R^\cdot)\neq 0\}$. Then $s=-\sup\{j\ |\
\mathrm{H}_{\mathfrak m}^j(A)\neq 0\}$. Since
 both $F^{\geq s}$ and $F^{\leq s-1}$ are minimal
free complexes, and $\mathrm{\underline{Ext}}_A^i(R^\cdot,k)_{\neq
-i}=0$ for any  $i \in \mathbb Z$,
\begin{equation}\label{3}
\mathrm{Ext.reg}\ F^{\geq s}= \left\{
\begin{array}{ll}
0 & F^{\geq s}~~\mbox{is not acyclic}\\
-\infty & F^{\geq s}~~\mbox{is acyclic}
\end{array}
\right.
\end{equation}
and
\begin{equation}\label{4}
\mathrm{Ext.reg}\ F^{\leq s-1}= \left\{
\begin{array}{ll}
0 & F^{\leq s-1}~~\mbox{is not acyclic}\\
-\infty & F^{\leq s-1}~~\mbox{is acyclic}.
\end{array}
\right.
\end{equation}
By the choice of $s$, $F^s \neq 0$ and thus
$\mathrm{\underline{Ext}}_A^{-s}(F^{\geq s},
k)=\mathrm{\underline{Hom}}_A(F^s, k) \neq 0$. Hence $F^{\geq s}$ is
not acyclic and $\mathrm{Ext.reg}\ F^{\geq s}=0$ by \eref{3}.  Let
$f$ be the quasi-isomorphism between $F^\cdot$ and $R^\cdot$. Then
$f$ induces naturally a morphism between $F^{\geq s}$ and $R^\cdot$,
denoted by $\tilde f$.

We claim that $\tilde f$ is a quasi-isomorphism. If $\tilde f$ is
not a quasi-isomorphism, then $\mathrm{cone}(\tilde f)\ncong 0$ in
$\mathrm{D^{b}}(\mathrm{Gr}\,A)$. Consider the following exact
triangle
$$F^{\geq s}\stackrel{\tilde f}{\longrightarrow}R^\cdot \longrightarrow \mathrm{cone}(\tilde f)\longrightarrow F^{\geq s}[1].$$
Since $R\mathrm{\Gamma}_{\mathfrak m}(R^\cdot)\cong A^\prime$ in
$\mathrm{D}(\mathrm{Gr}\,A^e)$, we have the following exact
sequences
\begin{equation}\label{5}
\left\{
\begin{array}{l}
0\rightarrow  \mathrm{H}_{\mathfrak m}^{-1}(\mathrm{cone}(\tilde f))
\rightarrow \mathrm{H}_{\mathfrak m}^0(F^{\geq s}) \rightarrow
\mathrm{H}_{\mathfrak m}^0(R^\cdot) \rightarrow
\mathrm{H}_{\mathfrak m}^{0}(\mathrm{cone}(\tilde f))
\rightarrow \mathrm{H}_{\mathfrak m}^1(F^{\geq s})\rightarrow 0,\\
\mathrm{H}_{\mathfrak m}^{j-1}(\mathrm{cone}(\tilde f))\cong
\mathrm{H}_{\mathfrak m}^{j}(F^{\geq s}),\quad j\neq 0, 1.
\end{array}
\right.
\end{equation}
It follows from the proof of Lemma \ref{quasi-isomorphism} that
$\mathrm{cone}(\tilde f) \cong h^{s-1}(\mathrm{cone}(\tilde
f))[1-s]$ in $\mathrm{D^{b}}(\mathrm{Gr}\,A)$. By the local duality
theorem, we have the following isomorphisms
\begin{eqnarray*}
R\mathrm{\Gamma}_{\mathfrak m}(\mathrm{cone}(\tilde f))^\prime
& \cong & R\mathrm{\underline{Hom}}_A(\mathrm{cone}(\tilde f), R^\cdot)\\
& \cong & R\mathrm{\underline{Hom}}_A(h^{s-1}(\mathrm{cone}(\tilde f))[1-s], R^\cdot)\\
& \cong & R\mathrm{\underline{Hom}}_A(h^{s-1}(\mathrm{cone}(\tilde
f)), R^\cdot)[s-1].
\end{eqnarray*}
By taking the $0^{th}$ cohomology modules, we have the following
isomorphism,
$$\mathrm{H}_{\mathfrak m}^0(\mathrm{cone}(\tilde f))^\prime\cong \mathrm{\underline{Ext}}_A^{s-1}(h^{s-1}(\mathrm{cone}(\tilde f)), R^\cdot).$$
Since $h^j(R^\cdot)=0$ for $j<s$, we have $\mathrm{H}_{\mathfrak
m}^0(\mathrm{cone}(\tilde
f))=\mathrm{\underline{Ext}}_A^{s-1}(h^{s-1}(\mathrm{cone}(\tilde
f)), R^\cdot)=0$. Since $\mathrm{CM.reg}\ F^{\geq
s}=\mathrm{Ext.reg}\ F^{\geq s}=0$, $\mathrm{H}_{\mathfrak
m}^{j-1}(\mathrm{cone}(\tilde f))_{>-j}$ $=0$ for any $j\neq 1$ by
\eref{5}. Thus $\mathrm{CM.reg}\ (\mathrm{cone}(\tilde f))\leq -1$
by Definition \ref{cm}. However, by Lemma \ref{quasi-isomorphism},
$F^{\leq s-1}$ is a minimal free resolution for
$\mathrm{cone}(\tilde f)$. Hence $\mathrm{Ext.reg}\
(\mathrm{cone}(\tilde f))=0$ by \eref{4}, which contradicts to
$\mathrm{CM.reg}\ (\mathrm{cone}(\tilde f))\leq -1$. Thus we have
proved that $\mathrm{cone}(\tilde f)$ is acyclic.

Since $F^{\leq s-1}$ is a minimal free resolution for
$\mathrm{cone}(\tilde f)$, then $F^j=0$ for $j\leq s-1$, which means
that the projective dimension of $R^\cdot$ is finite. By Theorem
\ref{Mori's Theorem}, $A$ is AS-Gorenstein. It follows from Lemma
\ref{CM reg of AS-Gorenstein algebra} and Theorem \ref{standard
Koszul AS-Gorenstein} that $A$ is AS-regular.
\end{proof}

\begin{remark}\label{koszul AS-regular algebras}
It is easy to show that (\rmnum 1)$\Leftrightarrow$(\rmnum 2) in
Theorem \ref{tim romer 4.1} always holds, even in non-commutative
case. By Proposition \ref {tim romer theorem 4.2}, $\mathrm{CM.reg}\
M-\mathrm{CM.reg}\ A=\mathrm{Ext.reg}\ M$ holds for all finitely
generated module $M$ with pd$_A M < \infty.$ Since there are lots of
non-Koszul AS-regular algebras generated in degree $1$ (e.g., see
\cite{as}), (\rmnum 1)$\Leftrightarrow$(\rmnum 3) does not hold in
general.
\end{remark}

\begin{proposition}\label{tim romer theorem 4.2}
Let $A$  be a Noetherian  connected graded $k$-algebra with a
balanced dualizing complex. If $M\in \mathrm{gr}\ A$ with
$\mathrm{pd}_A M < \infty,$ then
$$\mathrm{CM.reg}\ M-\mathrm{CM.reg}\ A=\mathrm{Ext.reg}\ M.$$
\end{proposition}
\begin{proof}[Proof] The proof in the commutative case (\cite[Theorem
4.2]{ro}) works well in our non-commutative case.
\end{proof}

\section*{Acknowledgments}

This research is supported by the NSFC (key project 10731070) and
supported by the Doctorate Foundation (No. 20060246003), Ministry of
Education of China.

\bibliography{}

\end{document}